\documentclass[11pt]{article}
\usepackage[utf8]{inputenc}
\usepackage{color}
\usepackage{subfigure}
\usepackage{url}
\usepackage{graphicx}
\usepackage{amsmath}
\usepackage{amsthm}
\usepackage{bm}
\usepackage[plainpages=false, bookmarks=false]{hyperref}
\usepackage{enumerate}
\usepackage{datetime}
\usepackage[linesnumbered,lined,boxed,commentsnumbered]{algorithm2e}

\usepackage{titlesec}
\titleformat{\section}
	{\normalfont\Large\bfseries\filcenter}{\thesection.}{1 ex}{}
\titleformat{\subsection}
	{\normalfont\normalsize\bfseries}{\thesubsection.}{1 ex}{}
\titleformat{\subsubsection}[runin]
	{\normalfont\normalsize\bfseries\filcenter}{\thesubsubsection.}{1 ex}{}

\usepackage[charter]{mathdesign}
\usepackage[mathcal]{eucal}

\usepackage[margin=1.2 in]{geometry}

\usepackage[square,numbers,sort&compress]{natbib}
\newcommand*{\doi}[1]{doi: \href{https://dx.doi.org/#1}{\urlstyle{rm}\nolinkurl{#1}}}
\newcommand*{\arxiv}[1]{arXiv:  \href{https://arxiv.org/abs/#1}{\urlstyle{rm}\nolinkurl{#1}}}

\usepackage{thmtools,thm-restate}
\declaretheorem[within=section]{theorem}
\declaretheorem[sibling=theorem]{lemma}
\declaretheorem[sibling=theorem,name=Proposition]{prop}
\declaretheorem[sibling=theorem]{corollary}

\declaretheorem[style=remark,sibling=theorem,qed={$\diamondsuit$}]{remark}
\declaretheorem[sibling=theorem,style=definition,qed={$\blacksquare$}]{definition}

\newcommand\cB{\mathcal{B}}

\newcommand\cM{\mathcal{M}}

\newcommand\cS{\mathcal{S}}

\newcommand\NN{\mathbb{N}}
\newcommand\QQ{\mathbb{Q}}
\newcommand\RR{\mathbb{R}}

\newcommand\ZZ{\mathbb{Z}}

\newcommand\nint{{\rm nint}}

\newcommand{\trans}[1]{{#1}^\top}

\newcommand{\defn}[1]{\emph{\color{blue} #1}} 

\newcommand\p{{\bf p}}

\newcommand\x{{\bf x}}
\newcommand\w{{\bf w}}
\newcommand\y{{\bf y}}

\renewcommand\r{{\bf r}}

\renewcommand{\l}{\ell}

\renewcommand\v{{\bf v}}

\newcommand\q{{\bf q}}

\DeclareMathOperator{\Lat}{Lat}

\newcommand{\bna}{\begin{eqnarray}}
\newcommand{\ena}{\end{eqnarray}}
\newcommand{\ba}{\begin{eqnarray*}}
\newcommand{\ea}{\end{eqnarray*}}

\newcommand{\eps}{\varepsilon}

\begin{document}
\title{Reconstruction in one dimension from 
unlabeled Euclidean lengths}
\author{Robert Connelly
\thanks{Department of Mathematics, Cornell University. \texttt{rc46@cornell.edu}. Partially supported by 
NSF grant DMS-1564493.}
\and 
Steven J. Gortler\thanks{School of Engineering and Applied Sciences, Harvard University. \texttt{sjg@cs.harvard.edu}. 
Partially supported by NSF grant DMS-1564473.}
\and Louis Theran\thanks{School of Mathematics and Statistics, University of St Andrews. \texttt{lst6@st-and.ac.uk}}}
\date{}
 
\maketitle 
\begin{abstract}
Let $G$ be a $3$-connected  ordered graph with 
$n$ vertices and $m$ edges. Let  
$\p$ be a randomly chosen mapping 
of these $n$ vertices 
to the integer range $\{1, 2,3, \ldots, 2^b\}$ for $b\ge m^2$. Let $\l$
be the vector of $m$ Euclidean lengths of $G$'s edges
under $\p$. In this paper, we show that,
with high probability over $\p$, we can efficiently reconstruct both 
$G$  and $\p$ from $\l$. 
This reconstruction problem is NP-HARD in the worst case, 
even if both $G$ and $\l$ are given.
We also show that our results stand in the presence of
small amounts of error in $\l$, and in the real setting,
with sufficiently accurate length measurements.

Our method combines lattice reduction, which has 
previously been used to solve random subset sum problems, 
with an algorithm of Seymour that can efficiently reconstruct
an ordered graph given an independence oracle for its matroid.
\end{abstract}

\section{Introduction}

Let $\p$ be a \defn{configuration} of $n$ points $p_i$ on the real line. 
Let 
$G$ be an ordered graph with $n$ vertices, 
numbered $\{1,\ldots, n\}$ (briefly, $[n]$) 
and $m$ edges, indexed by $[m]$.
Having edge indices is very useful when
applying linear algebraic methods to graphs.  We denote the edge with endpoints $i$ and $j$ as $ij$ or $\{i,j\}$. 
For each edge
$ij$ in $G$, we denote its Euclidean length as
$\ell_{ij}(\p):= |p_i-p_j|$. We denote by 
$\l=\ell(\p)$,
the $m$-vector of all of these
measured lengths.
In a \defn{labeled} reconstruction problem, 
we are given  $G$ and $\l$, 
and we are asked to reconstruct $\p$,
up to an unknowable translation and reflection.
Since the graph $G$ and the vector $\l$ are ordered,
this means we are given the bijection between lengths
and edges.
In an \defn{unlabeled} reconstruction problem, 
the input is the vector $\l$, and the 
task is to reconstruct $G$ and $\p$.
As in the unlabeled setting, $\p$ can be 
reconstructed only up to a translation and reflection.
Since $G$ is ordered, reconstructing it requires 
determining a bijection between the edges of $G$ and 
the entries of $\l$.  We note that the number of 
entries of $\l$ determines $m$, but not $n$, and that
$n$ is not in the input.  That problems of this type 
are difficult to solve in the worst case is classical.
Saxe \cite{saxe} showed in 1979 that the labeled version 
of the problem in one dimension is strongly NP-hard.

Two and three-dimensional
labeled and unlabeled reconstruction problems often arise
in applications such as
sensor network localization and molecular shape determination. These problems are even harder than the one-dimensional
setting studied here. But due to the practical importance of these
applications, 
many heuristics have been studied, especially in the labeled setting.
See~\cite{crippen,dattorro} for some surveys of this topic.
There do exist polynomial time
algorithms that have guarantees 
of success, but only  on restricted classes of graphs and configurations.
For example, 
labeled problems can be solved with
semidefinite programming if the 
unknown configuration $\p$ and graph $G$ 
correspond to a 
``universally rigid framework'' $(G,\p)$ \cite{so}.
Labeled problems in dimension $d$
can be solved with greedy algorithms
if the underlying graph 
happens to be ``$d$-trilateral''.  This means that
the graph has an edge set that allows one to localize one or more
$K_{d+2}$ subgraphs, and then iteratively 
glue vertices onto already-localized subgraphs 
using
$d+1$ edges, (and unambiguously glue together already localized subgraphs)
in a way to cover the entire graph.
There is also a huge literature on the related problem of low-rank 
matrix completion; 
see~\cite{LMRCsurvey} for a survey of this topic.
Typically, those results assume randomness 
in the choice of matrix sample locations, roughly corresponding to the graph 
$G$ in our setting.

Unlabeled problems are much harder, and there are fewer  approaches known to be viable.
In~\cite{BK1}, the authors prove that
for all dimensions $d$ and configurations $\p$ outside of a Zariski-closed subset of 
$d$-dimensional configurations, there is a unique way to associate
the values in 
$\l$ to the edges of the complete graph $K_n$.
Thus, when all pairwise distances are given, an exhaustive combinatorial 
search (exponential time) must correctly reconstruct
such configurations. 
For $d$-trilateral graphs, a backtracking approach called LIGA is presented in~\cite{dux1} and
a greedy approach called TRIBOND is presented in~\cite{dux2}.
TRIBOND is shown to run 
in $O(m^k)$ time, where $k$ is a fixed exponent for each
dimension. The work described by these groups has spearheaded 
the research in unlabeled reconstruction.
In~\cite{dux3}, 
the authors define a property for a 
 measured length vector
$\l$ they call 
``strongly generic''. 
They then say that the 
TRIBOND approach must succeed if 
$\l$
is strongly generic.
Notably, no sufficient condition is provided
for any  graph $G$ 
to have even a single strongly generic
measured length vector.
In~\cite{loopsTri} it is shown
that for $d \ge 2$, 
when the underlying graph 
is $d$-trilateral then the TRIBOND algorithm,
given exact measurements of a $d$-dimensional configuration $\p$, must 
correctly reconstruct $\p$, provided that $\p$ is not in an exceptional, Zariski-closed subset
of all $d$-dimensional configurations.
When $d=1$ the authors show that the graph must
be $2$-trilateral (instead of $1$-trilateral)
in order for a greedy trilateration method to work.

The unlabeled reconstruction problem in the one-dimensional 
integer setting 
also goes by names such as ``the turnpike problem''
and the ``partial digest problem''.  Theoretical work
on these problems typically uses worst case analysis and often only 
applies to the complete graph
(e.g.~\cite{lem1,ski1,cie1}).

In this paper we  present  an algorithm that 
can provably solve the one-dimensional unlabeled (and thus, also the labeled)
reconstruction
problems in polynomial time, for each fixed ordered graph, 
with high probability over randomly chosen 
$\p$, as long as we have  
sufficiently accurate length measurements. 
We only require that $G$ is 
$2$-connected for the labeled case and 
$3$-connected for the unlabeled case.
As we will see below, these conditions are
necessary for any method to work.

The algorithm is based on combining the LLL
lattice reduction algorithm~\cite{LLL,LO,frieze} 
together with the graphic matroid reconstruction
algorithm of Seymour~\cite{sey}.
The limiting factor of 
our algorithm in practice is its
reliance on very precise measurements of 
the edge lengths  $\l$. But the existence of  an 
efficient algorithm with
these theoretical guarantees was surprising to us.
Even though the required precision grows with the 
number of edges, we give an explicit bound.  
This is in contrast with
theoretical
results about generic configurations  often found in rigidity theory,
which do not apply to configurations with integer coordinates or 
inexactly-measured lengths.

\subsection{Connectivity assumption}
Clearly we can never detect a translation or reflection of $\p$
in the line, so we will only be interested in recovery of the configuration up
to such \defn{congruence}. 
In the labeled setting, if $i$ is a cut vertex of $G$, then
$G$ is the union of two graphs $H_1$ and $H_2$ with at least 
two vertices each, so that the vertex set of $H_1$ and $H_2$
intersect on exactly $\{i\}$.  Hence, we can reflect the positions of the 
vertices in $H_1$ across $p_i$
without changing $\l$. Such a flip is invisible to
our input data.  To avoid such partial
reflections, we must assume that $G$ is $2$-connected.

The question of whether it is possible to recover $\p$, a configuration in $\RR^d$, uniquely
up to isometry from the the labeled length measurements 
is known in rigidity theory as the question of whether 
the \defn{framework} $(G,\p)$ is \defn{globally rigid} in $d$-dimensions.  If, for generic 
$\p$, $(G,\p)$ is always globally rigid in dimension $d$, 
then $G$ is called \defn{generically globally rigid} in 
dimension $d$  \cite{conGR,ght}.  
It is a folklore result that $G$ is 
generically globally rigid in dimension one if and only 
if it is $2$-connected. 

In higher dimensions, $(d+1)$-connectivity 
is necessary for a graph to be generically globally rigid \cite{hendrickson}, 
but it is not sufficient \cite{connelly}.  In dimension two, there is a combinatorial characterization 
of generically globally rigid graphs \cite{JJ}.  In dimensions $d\ge 3$, 
there is no known combinatorial characterization of generically globally rigid 
graphs, but there is a characterization that can be checked with 
a randomized algorithm that uses linear algebra \cite{ght}.

\begin{figure}
\begin{center} \begin{tabular}{cccc}
  {  \includegraphics[scale=.16]{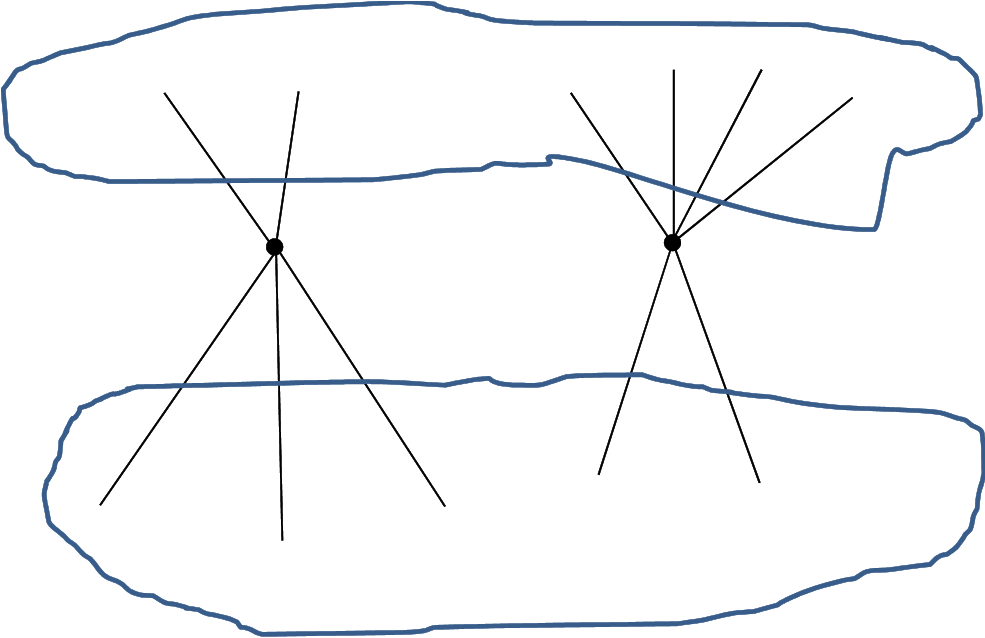}}&
  {  \includegraphics[scale=.16]{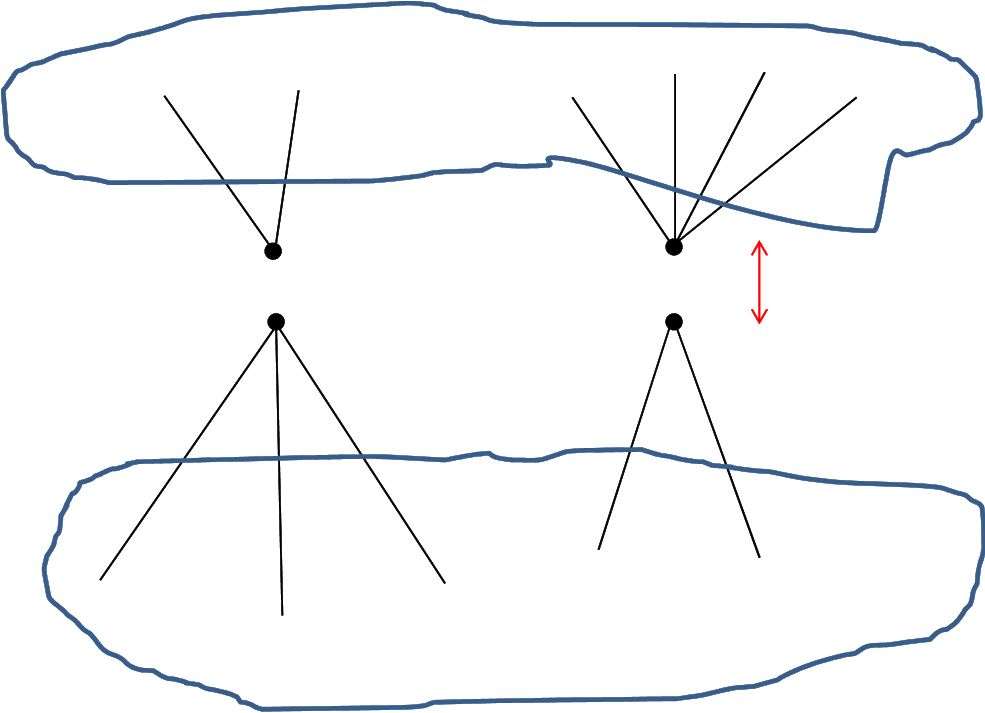}}&
  {  \includegraphics[scale=.16]{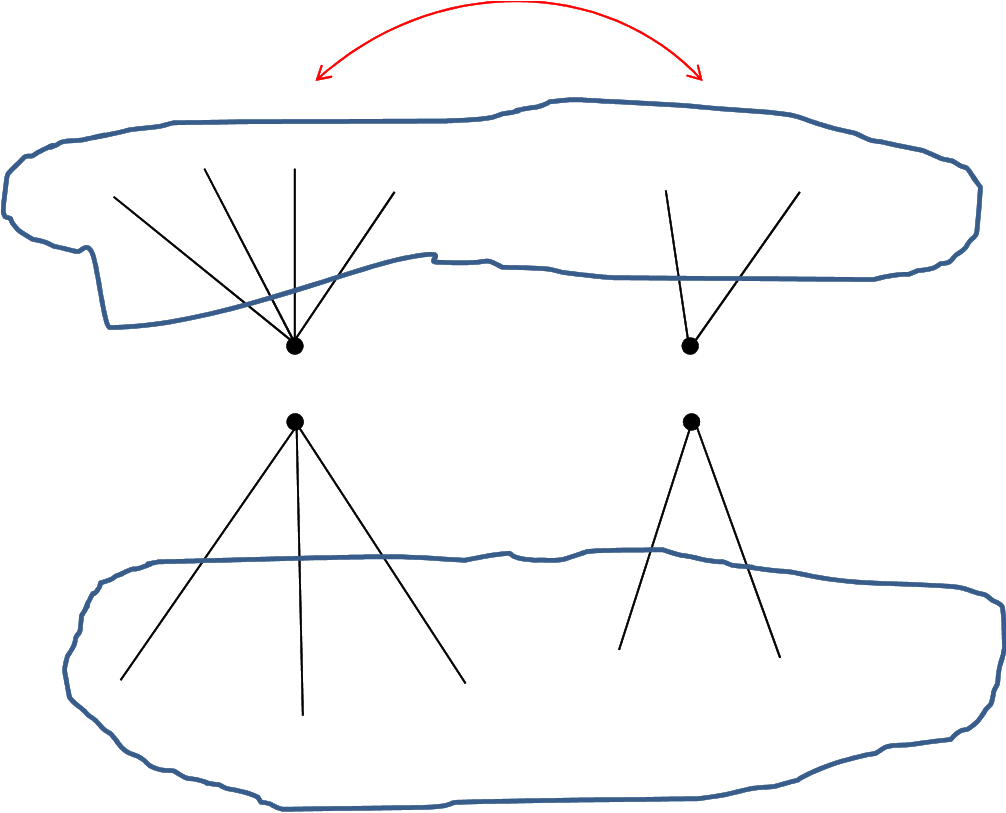}}&
  {  \includegraphics[scale=.16]{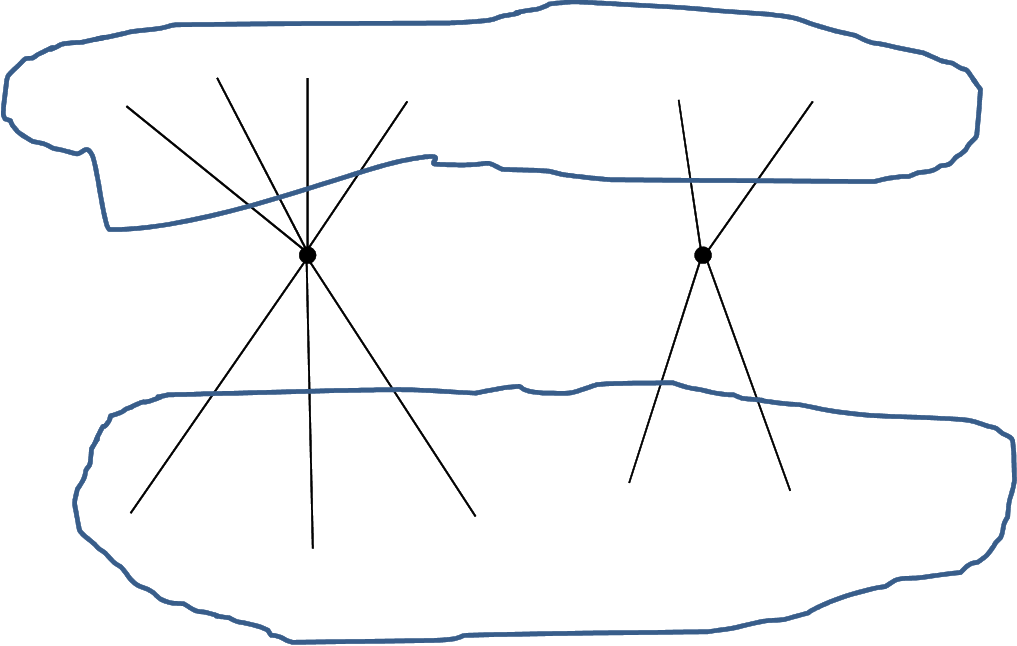}}\\
  (a) & (b) &(c)&(d) \\
\end{tabular} \end{center}
  \caption{The graphs, (a) and (d) are 
  related by a $2$-swap.
Note that the edge lengths in one dimension
are unchanged under a
2-swap.}
  \label{fig:reverse}
\end{figure}

\vspace{.1in}
In the unlabeled setting, 
we can never detect a renumbering of the 
points in $\p$ along with a corresponding vertex renumbering in $G$, so we will
be unconcerned with such isomorphic outputs.
More formally,
let $G = (V,E)$ be an ordered graph with vertex set $[n]$ and with an ordering of the edges given by 
a bijection $\sigma : E\to [m]$.
If $H = (V,E')$ is another ordered graph with the same vertex set, such that there is a 
bijection $\pi : V\to V$ such that $ij\in E$ if and only if $\pi(i)\pi(j)\in E'$ and the edges 
of $H$ are ordered via a bijection $\tau : E'\to [m]$ so that 
$\sigma(ij) = \tau(\pi(i)\pi(j))$, then $G$ and $H$ are isomorphic as ordered graphs, and 
it is impossible to detect whether a measurement vector 
$\l$ arose from $G$ or $H$.

Without labeling information, if $G$ has a 
cut vertex $i$, then, decomposing $G$ into $H_1$ and 
$H_2$ as above, we cannot hope to distinguish measurements 
from $G$ and those from a graph made from identifying $H_1$
and $H_2$ along some other vertex besides $i$, because besides 
reflecting the positions of vertices in $H_2$ across $i$, 
we can also translate them.  Even if $G$ 
is $2$-connected but not $3$-connected, the unlabeled 
problem does not have a unique solution.  Let $\{i,j\}$
be a cut pair and decompose $G$ into graphs $H_1$ and $H_2$
so that their vertex sets intersect exactly on $\{i,j\}$.
Any set of length measurements of $G$ also arises from 
the graph obtained by identifying the copy of $i$ in $H_1$
with the copy of $j$ in $H_2$ and vice versa.  This operation, 
illustrated in Figure~\ref{fig:reverse}, is called a \defn{$2$-swap}.
Graphs that are related by a $2$-swap are not isomorphic, but 
they have the same graphic matroids.
To avoid ambiguities such swapping, we must assume, at least, 
that $G$ is $3$-connected.
There are a number of
subtleties in how to even define an appropriate notion
of generic global rigidity in the unlabeled setting.
These are discussed in~\cite{gugr, loopsTri}.

\subsection{Integer setting}
Moving to an algorithmic setting, we 
will now assume that 
the  points
are placed at integer locations in the set $[B]$, 
with $B=2^b$ for some positive integer
$b$. 
For the labeled problem, we will want an algorithm that
takes as input $G$ and an integer vector $\l\in [B]^m$, and outputs $\p$ (up to congruence).
For the unlabeled problem, we will want an algorithm that
takes as input the integer vector $\l$, and outputs $G$ (up to vertex relabeling) and $\p$ (up to congruence).

Even the labeled problem is strongly NP-HARD by a result of Saxe~\cite{saxe}. 
Interestingly, one
of his reductions is from SUBSET-SUM. 
This is a problem that, while hard in the worst case,
can be efficiently solved,
using the Lenstra-Lenstra-Lovász (LLL) lattice reduction 
algorithm \cite{LLL}, 
for most instances that have sufficiently large integers~\cite{LO}.
Indeed, we will take cue from this insight and use LLL to solve both the
labeled and unlabeled reconstruction problems.

We say that a sequence, $A_m$, of 
events in a probability space occurs \defn{with 
high probability as $m \to \infty$} (WHP)  
if $\Pr[A_m] = 1 - o(1)$.
Here is our basic result.  
\begin{theorem}
\label{thm:main1}
Let $G$ be a 
$3$-connected (resp. $2$-connected) graph with $m$ edges.
Let $b \ge m^2$. Let each $p_i$ be chosen independently and 
uniformly at random from $[2^b]$.
Then there is a polynomial time algorithm that succeeds, 
WHP, on  the unlabeled (resp. labeled) reconstruction problems.
\end{theorem}

The success probability, the convergence 
rate and the minimal required $b$
can be estimated more carefully (see 
Proposition~\ref{prop:step1} and
Remark~\ref{rem:err3} below).
In any case, what we can prove seems to be pessimistic.
Experiments (see below) show that
our algorithm  succeeds with very high probability on randomly 
chosen 
$\p$ even with  significantly smaller $b$.
 In part, this is because the 
LLL algorithm very often  outperforms its worst case behavior~\cite{ave}
on random input lattices.
On the other hand, 
lattice reduction is a brittle operation, and, for fixed $m$, it will fail
ungracefully once $b$ becomes too small.

\subsection{Errors in input}
To move towards a numerical setting, we need to differentiate between the 
geometric data in $\p$ and the combinatorial data that records the 
order of the points of $\p$ along the line.  Since we only have access to 
the measurements $\l$, the ordering data will be represented by an 
orientation $\sigma_\p$ of a graph $G$ produced by the algorithm;
the orientation of an edge indicates the ordering of its endpoints on 
the line using the configuration $\p$.
To model error or noise, we consider the case where
there is $\{-1,0,1\}$ error, added adversarially,  to each of the integer measurements 
$\ell_{ij}$. Relative to our scale of $2^{m^2}$, a $\{-1,0,1\}$ error 
is very small.

\begin{theorem}
\label{thm:main2}
Let $G$ be an ordered graph with $m$ edges. 
For the labeled setting, assume 
$G$ is $2$-connected.
For the unlabeled setting, assume $G$ is $3$-connected.
Let $b \ge m^2$. Let each $p_i$ be chosen 
independently and uniformly at random from $[2^b]$.
Assume that 
$\{-1,0,1\}$ noise has been added
adversarially to the $\ell_{ij}$.
Then, in both the labeled and unlabeled settings,
there is a polynomial time algorithm that, WHP, 
correctly outputs $G$ and $\sigma_\p$, 
up to reversing the orientation of every edge in $\sigma_\p$.
\end{theorem}
The bound on $b$ given here can be improved slightly, at the cost of a 
more complicated statement.  Remark \ref{rem:err3} describes 
such an improvement. 
In Section~\ref{sec:fail} we discuss ways in which our algorithm can 
fail.

Once we have the combinatorial data of $G$ and $\sigma_\p$, it is 
easy to reconstruct $\p$, up to congruence, when $\l$ has no error.
We can simply use a spanning tree of $G$ to lay out the $p_i$ one at a 
time.  If there are errors, then it is unlikely that there is 
any configuration corresponding to the observed measurements.  However, 
we can apply a linear least squares algorithm to get an 
estimate of $\p$ with small error.
\begin{theorem}
\label{thm:main2b}
Let $G$ be a graph with $m$ edges, and suppose that 
we are given measurements $\l$ of an unknown 
configuration $\p$, with adversarial noise
in the model of Theorem \ref{thm:main2}, and 
an orientation $\sigma_\p$ of $G$ that is 
consistent with the ordering of the points in $\p$ on the 
line.

Then we can efficiently
compute a vector $\p^* \in \QQ^n$ so that, 
there is a vector $\q = \pm \p + t{\bf 1}$ with
\[
\| \p^*-\q\| \le n^{2.5}\sqrt{m}
\]
where $\pm \p$ represents either $\p$ or the configuration 
$-\p$ obtained from negating each of the points $p_i$ in $\p$ 
and $t{\bf 1}$ represents translation of $\p$ 
by an unknown rational number $t$.
\end{theorem}
Since we are working with signals
at the scale of $B=2^{m^2}$, this is a very small error indeed.
We show in Remark~\ref{rem:err1} shows that 
when we have noise of larger magnitude than
$1$, the best approach is to round and truncate the input 
in such a way that the errors on each length measurement 
are in the set $\{-1,0,1\}$, albeit with a smaller $b$. 
As long as this new, smaller, $b$ is sufficiently large, then our theorem will still apply.
Informally, the important thing is not the number of error bits, 
but the number of signal bits.
Additionally, Remark~\ref{rem:err2} 
shows that with sufficiently large $m$,
we can tolerate
any constant error 
maintained in our representation,
or even a number of error bits
that is bounded by a function that
is $o(m)$.

\subsection{Real setting}
\label{sec:real}

Theorem~\ref{thm:main2} 
can be used  in the setting where $\p$ is
real valued.
We suppose that each $p_i$ 
is chosen independently from the uniform distribution on the 
closed interval $[1,B]$, which determines the underlying  $\ell_{ij}$.
We also suppose that each $\ell_{ij}$ is corrupted by an adversarial error
$\eps_{ij}$ with 
$-\frac{1}{2} \le \eps_{ij} \le \frac{1}{2}$, so 
that the input data is 
\[
    \ell_{ij} + \eps_{ij}.
\]
Let $\ell'_{ij}:= |\nint(p_i)-\nint(p_j)|$,
where $\nint$ is the nearest integer. 
Under these assumptions
we have 
\ba
|\ell'_{ij} - \nint(\ell_{ij}+\eps_{ij}) |\le 1.
\ea
As a result, 
the rounded data,  $\nint(\ell_{ij}+\eps_{ij})$,
will  be a $\{-1,0,1\}$ noisy version 
of the lengths obtained from the rounded configuration,
$\nint(\p)$. Hence, our Theorem~\ref{thm:main2} applies to this integer
input.  If we instead start with a real configuration
$\p$ with points lying in the 
closed interval $[0,1]$, scaling by $B$ and replaying the 
above construction shows that we can tolerate 
adversarial error on the order of $\frac{1}{2^{m^2}}$.

If we are given data with higher $\eps$ values, 
we can simply divide by an appropriate constant  so that, afterwards,  
$-\frac{1}{2} \le \eps_{ij} \le \frac{1}{2}$, and then proceed with a smaller effective $b$
and $B$. As long as $b$ is large enough, 
Theorem~\ref{thm:main2} will still apply.
In light of Remark~\ref{rem:err1}, this approach
is strictly better than keeping around additional
error bits.
Though in light of Remark~\ref{rem:err2}, our theorem can 
tolerate $o(m)$ extra  
error bits.

\subsection{Basic ideas}
Consider traversing a cycle $C$ in $G$, which is an ordered list
of distinct vertices (except for the first and last, which are equal)
\[
    v_1, v_2, \ldots, v_k =v_1
\]
so that the edges $v_{i}v_{i+1}$ and $v_{k-1}v_1$ are in $G$.
As we traverse $C$, we consider the configuration $\p$. 
Some of the edges go to the right ($p_{v_{i+1}} > p_{v_i}$) 
and some to the left ($p_{v_{i+1}} <  p_{v_i}$). 
If we sum up the Euclidean lengths of these edges, 
with a $+1$ coefficient for edges that go to the right 
and and a $-1$ coefficient for edges that go to the left, 
the sum is zero, since the cycle traversal starts and ends 
at $p_{v_1}$.  Thus, the cycle $C$ determines an
integer relation on the coordinates of the length vector $\l$.
The integer coefficients in  this relation are all ``small'', 
(the actual bounds we use  will become apparent in Section~\ref{sec:lll})
in fact only $0$, $-1$, and $1$, and, hence, so is the norm of 
the coefficient vector for such a cycle relation.

Let us call $\sigma_\p$ the ordering of the vertices 
along the line.  Repeating the argument for each 
cycle in $G$ defines more small integer relations on 
the coordinates of $\l$.  
We define $S^{\sigma_\p}$ 
to be the linear span, in $\RR^m$, of all the 
coefficient vectors of integer relations defined by cycles 
in $G$.
The space $S^{\sigma_\p}$ contains all the 
small integer relations that are forced on $\l$ because it 
arises from measuring $\p$.  Heuristically, if 
$\p$ is is selected uniformly at random, then 
$\l$ should not satisfy \emph{any other} small integer 
relation, because $\l$ still behaves otherwise randomly.

In Section \ref{sec:lll}, we make this heuristic 
reasoning precise.  What we show is that, if 
 $\p$ is selected uniformly 
at random from configurations of $n$ integer 
points in $[0,B]$, then, WHP any integer relation
that is satisfied by the coordinates of $\l$ and 
has norm below a cutoff value is in $S^{\sigma_\p}$.
This gives us a way to find a basis for $S^{\sigma_\p}$,
provided that we can find a basis for an appropriate lattice 
with small enough basis vectors.

Finding a lattice basis  consisting of vectors 
that are as short as possible is an NP-hard problem.
However, the cutoff we need is large enough so that 
any LLL reduced basis for the lattice we construct will, 
WHP, allow us to find a basis for $S^{\sigma_\p}$.
Our analysis builds on work of 
Lagarias and Odylyzko \cite{LO} and 
Frieze \cite{frieze} on random SUBSET-SUM problems.
The main technical difference between our problem and random 
SUBSET-SUM instances is that $\l$ is obtained 
from $\p$ via a non-linear mapping.  We handle this issue 
using a counting argument that extends those of Frieze \cite{frieze}.  Because the LLL algorithm 
\cite{LLL} can find a reduced basis in polynomial 
time, this step of the algorithm is polynomial.  

Once  we recover $S^{\sigma_{\p}}$, our next task 
is to use it to recover the ordered graph $G$. For this, we will use a 
polynomial time algorithm of Seymour~\cite{sey} that constructs an 
unknown graph using an independence oracle for its graphic matroid.  
We can implement the oracle in polynomial 
time using the basis for $S^{\sigma_\p}$ found in the 
previous step, so we can recover $G$ in polynomial time.
To complete the combinatorial steps, we need to recover the 
orientation $\sigma_\p$.
Using $S^{\sigma_\p}$, we can compute a consistent orientation 
for the edges of any cycle in $G$.  We then find a globally 
consistent orientation, which must be $\pm \sigma_{\p}$,
using an iterative procedure based on an ear decomposition of $G$.  In this step, 
we use $2$-connectivity of $G$, because otherwise an ear 
decomposition will not exist.

Finally, we can use  $G$, $\sigma_{\p}$  and $\l$ 
to determine the positions of each of the $p_i$.
In the noiseless case, this can be done by traversing a spanning 
tree of $G$ and laying out the points. If there is noise, 
we solve a least squares problem to approximate $\p$.

\begin{algorithm}[H]
\SetKwInOut{Input}{input}\SetKwInOut{Output}{output}
\Input{$\l$}
\Output{($G$, $\p$)}
\BlankLine
$S^{\sigma_{\p}} \leftarrow$ ComputeRelations($\l$)\;
$G \leftarrow$ Seymour($S^{\sigma_{\p}}$)\;
$\sigma_\p \leftarrow$  ComputeOrientation($G$, $S^{\sigma_{\p}}$)\;
$\p \leftarrow$ Layout($G$, $\sigma_\p$, $\l$)\;
\Return ($G$,$\p$)\;
 \caption{The full Algorithm}
\end{algorithm}

The 
methods in this paper do not 
directly generalize to higher dimensions.  In fact, 
each of our steps poses a challenging problem once
the dimension becomes at least two.  There is 
no direct analogue to the $\pm 1$ cycle relations, and
instead, we would have to find integer polynomial relations 
on the measured lengths. To this end, we have no known non-trivial bound on 
either the degree or coefficient-sizes of these polynomials.  Even 
if it were possible to recover the polynomials, and then, somehow,
using a (so-far unknown) algorithmic version of \cite{gugr}, 
the unknown graph $G$, the remaining labeled graph realization 
problem is difficult in practice.

\subsection{Related LLL Work}
Recently in ~\cite{gam2}, a  similar approach  has been carefully
applied  to recover an unknown discrete signal from multiple randomly chosen 
large integer linear measurements. 
As described in~\cite{gam2}, an algorithm based on LLL lattice reduction 
can even work in the presence of adversarially chosen error on the 
input, provided that the error is 
sufficiently small.   Theorem 2.1 of \cite{gam2}
provides a detailed analysis of the relationship between the various
parameters.

In other related recent work, 
\cite{andoni} has shown that the LLL
method can be used to efficiently solve 
the correspondence retrieval problem
and the phase retrieval problem with 
no noise. 
Song et al.~\cite{song}
have shown how LLL can be used to
efficiently learn
single periodic neurons in the low noise
regime (this problem also generalizes  phase retrieval).
Work in~\cite{zlat,diak} shows how a number of other problems in machine learning,
such as learning mixtures of Gaussians,
can be solved by LLL, provided that the noise is sufficiently 
small.
The main difference between our setting and those 
of~\cite{LO,gam2,andoni,song,zlat,diak}, is that 
in our case, we are looking for 
a number of  linearly independent small relations in one random 
data vector, instead of  
a single small relation in one or more random data vectors.

In~\cite{LO,gam2}, the data that satisfies a small integer relation ($\l$ in our setting)
is directly generated by a random process.
In our work, as in the work of~\cite{andoni,song,zlat,diak}, 
the input data is obtained indirectly from the random process.
And in particular, in this paper, the process going from the randomly 
chosen configuration $\p$ to the length measurements $\l$
involves an absolute value operation.

In~\cite{gam2} and~\cite{diak}, the authors show that when one
has  $O(m)$ random data vectors that satisfy a common small integer
relation, then LLL can be successfully applied with only 
around $m$ bits of reliable data, instead of the $m^2$ bits that we need here.
For our reconstruction problem, 
we only have one underlying
random configuration $\p$, so we cannot exploit any additional linear 
relations to get a quantitative improvement.

In our random model, each $p_i$ is chosen independently from a uniform
distribution. In rigidity applications, this is a natural assumption 
to make. Presumably, similar results could hold under some other
sampling assumptions for the points $p_i$.
In particular, the works in~\cite{andoni,song,zlat,diak} consider a 
Gaussian random process. What appears to be key is that 
the sampling distribution possesses some sufficient
anti-concentration properties~\cite{zlat,diak}.

\section{Graph theoretic preliminaries}
First we set up some graph-theoretic background.  

\subsection{Graphs, orientations, and measurements}
In this paper, $G = (V,E)$ will denote a simple, undirected ordered
graph with $n$ ordered vertices and $m$ 
ordered edges.  
These orderings will allow us to index various
matrix rows and columns associated with edges and
vertices.
\begin{definition}\label{def: graphs}
Let $G = (V,E)$ be an undirected simple 
graph with $n$ 
vertices and $m$ edges.  We  denote undirected edges 
by $\{i,j\}$, with $i, j\in V$.

An \defn{orientation} $\sigma$ of $G$ is a  map
that assigns each  undirected edge $\{i,j\}$ of $G$
to exactly one of the directed edges $(i,j)$ or $(j,i)$; i.e.,
$\sigma(\{i,j\})\in \{(i,j),(j,i)\}$.   
Informally, $\sigma$ assigns one endpoint of each edge to be 
the tail and the other to be the head.

An \defn{oriented graph} $G^\sigma$ is a directed 
graph obtained from an undirected graph $G$ and 
an orientation $\sigma$ of $G$ by replacing each 
undirected edge $\{i,j\}$ of $G$ with $\sigma(\{i,j\})$.
\end{definition}

We are usually interested in a pair $(G,\p)$ where
$G$ has $n$ vertices and $\p$ is a configuration 
of $n$ points $(p_1, \ldots, p_n)$ on the line.
Such a pair has a natural orientation of 
$G$
associated with it.

\begin{definition}\label{def: point ordering}
Let $(G,\p)$ denote a graph with $n$ vertices
and $\p$ a configuration of $n$ points.  Define
the \defn{configuration orientation}
$\sigma_\p$ by 
\[
    \sigma_\p(\{i,j\}) = 
    \begin{cases}
        (i,j) & \text{if $p_i < p_j$} \\
        (j,i) &\text{if $p_i \ge p_j$}    \end{cases}
    \qquad
    \text{(for all $\{i,j\}\in E$)}
\]
We say that an orientation $\sigma$ of $G$
is \defn{vertex 
consistent} if $\sigma = \sigma_\p$ for some configuration $\p$.
\end{definition}
There are $2^m$ possible orientations, but, when $m > n$, 
many fewer are vertex consistent.
\begin{lemma}\label{lem: point ordering count}
Let $G$ be a graph with $n$ vertices.  There are at most $n!$ 
vertex consistent orientations of $G$.
\end{lemma}
\begin{proof}
The orientation $\sigma_\p$ is determined by the 
permutation of $\{1, \ldots, n\}$ 
induced by the ordering of the $p_i$ on the line.
\end{proof}

We now introduce some specialized facts about 
a standard linear representation of a graphic 
matroid over $\RR$.  We start by associating 
two vector configurations with a graph $G$.
\begin{definition}\label{def: incidence matrices}
Let $G$ be a graph with $n$ vertices and $m$ 
edges. 
For $i\ne j$, both in $V$, 
a \defn{signed edge incidence vector} $e_{(i,j)}$
has $n$ coordinates indexed by $V$; all the 
entries are zero except for the $i$th, 
which is $-1$ and the $j$th, which is $+1$.  (So that
$e_{(i,j)} = - e_{(j,i)}$.)

Given an ordered graph $G$ and an orientation 
$\sigma$ of $G$,
the 
\defn{signed edge incidence matrix}, $M^\sigma(G)$,
has as its rows the vectors $e_{\sigma(\{i,j\})}$,
over all of the edges $\{i,j\}$ of $G$.

Let $\rho = v_1, v_2, \ldots, v_k = v_1$
be a cycle in $G$.  The \defn{signed 
cycle vector} $\w_{\rho,\sigma}$ 
of $\rho$ in $G^\sigma$ has $m$ 
coordinates indexed by $E$. Coordinates corresponding 
to edges not traversed by $\rho$ are zero.  
The coordinate corresponding to the edge $\{v_{i}, v_{i+1}\}$
is $1$ if $\sigma(\{v_{i}, v_{i+1}\}) = (v_{i}, v_{i+1})$
and otherwise it is $-1$.

If we have fixed some $\p$, we  
can shorten  $\w_{\rho,\sigma_{\p}}$
to $\w_{\rho}$ (with no explicit orientation written).
\end{definition}
It is easy to see that a signed cycle vector
$\w_{\rho,\sigma_{\p}}$ must be in the
kernel of the transpose of $M^\sigma(G)$.
Now we shift focus to an associated linear space.
\begin{definition}
Let $G$ be a graph with $n$ vertices and $m$ edges and $\sigma$
an orientation.  An 
\defn{$\RR$-cycle} is a vector in the kernel of 
the transpose of $M^\sigma(G)$.  The \defn{$\RR$-cycle space} (\defn{cycle space}, 
for short) $S^\sigma\subseteq \RR^m$ is the subspace of all 
$\RR$-cycles. 
A basis of $S^\sigma$ is \defn{induced by graph cycles} if it 
contains only signed cycle vectors.
\end{definition}
For our intended application, it is important that $S^\sigma$ have a 
basis comprising integer vectors of small norm.  It is a classical 
fact that this is the case.
\begin{prop}[{See e.g, \cite[Chapter 5]{oxley}}]\label{prop: matroid basics}
Let $G$ be a connected ordered graph with $n$ vertices and $m$ edges and $\sigma$
an orientation of $G$.   The $\RR$-cycle space $S^\sigma$, 
has  dimension $c := m - n + 1$.  Moreover, it is spanned by the 
signed cycle vectors 
in $G^\sigma$.  In particular, 
$S^\sigma$ has  a basis induced by graph cycles.
\end{prop}
If we denote by $-\sigma$ the orientation of $G$ obtained by 
reversing every edge in $G^\sigma$, it is easy to 
see that $S^\sigma = S^{-\sigma}$.
We now describe a family of linear maps
related to edge length measurements.
\begin{definition}\label{def: measurements}
Let $G$ be an ordered graph with $n$ vertices and $\sigma$
an orientation of $G$.  We define the \defn{oriented 
measurement map} $\ell_\sigma : \RR^n \to \RR^m$ by 
\[
    \ell_\sigma(\p)_{\{i,j\}} := 
    \begin{cases}
        p_j - p_i & \text{if $\sigma(\{i,j\}) = (i,j)$} \\
        p_i - p_j & \text{if $\sigma(\{i,j\}) = (j,i)$}
    \end{cases}
\]
For any choice of $\sigma$, $\ell_\sigma$ is a linear map.
We also define the \defn{measurement map} $\ell$ by 
\[
    \ell(\p)_{\{i,j\}} = |p_j - p_i|
\]
The measurement map is only piecewise linear.
\end{definition}

\begin{lemma}\label{lem:stress}
Let $(G,\p)$ be an ordered graph with $n$ vertices and $\p$ a
configuration of $n$ points.  Let $\rho$ be a  cycle of
$G$. Let $\w_{\rho}$ be its signed cycle vector in
$G^{\sigma_\p}$. 
Then 
\[
    \trans{\w}_\rho \ell(\p) = 0
\]
\end{lemma}
\begin{proof}
Unraveling definitions, we see that $\ell(\p) = \ell_{\sigma_\p}(\p)$, and that, in the standard bases for the domain 
$\RR^n$ and codomain $\RR^m$, respectively, the matrix of the linear transformation $\ell_{\sigma_\p}$ is $M^{\sigma_p}(G)$.  
We then compute 
\[
     \trans{\w}_\rho \ell(\p) = \trans{\w}_\rho \ell_{\sigma_\p}(\p) = \trans{\w}_\rho M^{\sigma_p}(G) \p = 0
\]
because $\trans{\w}_\rho$ is in the cokernel of $M^{\sigma_p}(G)$ by definition.
\end{proof}

\section{The algorithm}\label{sec: alg}
In this section we describe the algorithm in 
detail and develop the tools required for its 
analysis. We will focus on the harder, unlabeled 
setting, and afterwards describe the various simplifications
that can be applied in the labeled case.

\subsection{The sampling model}
The algorithm has the following parameters.  We have
a precision $b\in \NN$.  There 
is fixed ordered graph $G$ with $n$ vertices and $m$ 
edges.  
Our algorithm takes as input $\l = \ell(\p) + \eps$,
where $\p$ is selected uniformly at random among 
configurations of $n$ points with coordinates 
in $[2^b]$ 
and $\eps$ is an error vector in $\{-1,0,1\}^m$.

The algorithm does not have access to $n$, only $m$
(via the number of coordinates in $\l$).  All the 
probability statements are taken with respect to $\p$,
which is not part of the input.
This complicates the probabilistic analysis a little
bit relative to \cite{LO,frieze}, but much 
remains similar.

\subsection{Reconstruct $S^{\sigma_\p}$}
\label{sec:lll}

In this section, we will show how, WHP, to reconstruct
$S^{\sigma_\p}$ from $\l$. Our main tool
will be the LLL algorithm~\cite{LLL}, as used
by~\cite{LO}, and analyzed by~\cite{frieze}.

\begin{definition}
For any vector $\r$ in $\ZZ^{m}$
define the \defn{lattice} $\Lat(\r)$ 
in $\ZZ^{m+1}$  as the integer span of
the columns of the following 
$(m+1)$-by-$m$ \defn{lattice generating matrix}

\ba
L(\r):=
\begin{pmatrix}
I_m \\ \trans{\r}
\end{pmatrix}
\ea
where $I_m$ is the $m\times m$ identity matrix.
\end{definition}

\begin{remark}
Frieze~\cite{frieze} multiplies the last row
of $L(\r)$ by a large constant. We do not do that,
so that we can easily deal with
noise. In~\cite{noConst}, we elaborate on this issue
and show how, in the original context of SUBSET-SUM,
Frieze's analysis goes through  without the use
of the large constant.
\end{remark}

In our algorithm we will work over the lattice
$\Lat(\l)$, where $\l$ is the edge length 
measurements of the unknown, configuration $\p$, 
plus adervarially chosen noise.  Figure \ref{alg:cr}
summarizes the algorithm in pseudo-code.

\begin{algorithm}[H]
\label{alg:cr}
\SetKwInOut{Input}{input}\SetKwInOut{Output}{output}
\Input{$\l$}
\Output{$S^{\sigma_{\p}}$}
\BlankLine
$m \leftarrow |\l|$\;
$\Lambda \leftarrow$ LLL($\Lat(\l$))\;
basis $\leftarrow$ threshold($\Lambda$,$\sqrt{2m}2^{m/2}$)\;
$S^{\sigma_{\p}} \leftarrow$ truncate(basis)\;
\Return ($S^{\sigma_{\p}}$)\;
 \caption{The ComputeRelations Algorithm}
\end{algorithm}

If $\x$ is a vector in $\RR^m$ and $a\in \RR$ a real number, 
we denote by $\begin{pmatrix}
    \x \\ a
\end{pmatrix}$ the vector in $\RR^{m+1}$ that has $\x$ as its 
first $m$ coordinates 
and $a$ as its $(m+1)$st. We call $\x$ the \defn{truncation} 
of $\begin{pmatrix}
    \x \\ a
\end{pmatrix}$. 
Now we can describe the algorithm to reconstruct $S^{\sigma_\p}$ in more detail. 
\begin{itemize}
    \item Given, $\l$, we let $m$ be the number of coordinates in $\l$.  This is the number of 
        edges of $G$.  The parameter $m$ determines a threshold value $\sqrt{2m}2^{m/2}$.
    \item We then use the LLL algorithm to compute a reduced basis $\Lambda$ for the lattice 
        $\Lat(\l)$.
    \item Finally, we truncate every vector in $\Lambda$ of norm at most the threshold value $\sqrt{2m}2^{m/2}$
        and return these as the basis for $S^{\sigma_\p}$.
\end{itemize}
The rest of the section contains the analysis of this algorithm. Fix 
an ordered graph $G$, a configuration $\p$ and an error 
vector $\eps$, which determine the 
observed measurements $\l = \ell(\p) + \eps$.
Because $\l$ arises from measuring an integer configuration on the line, 
$\Lat(\l)$ will deterministically contain some short vectors associated 
with generators of the $\RR$-cycle space $S^{\sigma_\p}$.  
\begin{definition}
\label{def:hatS}
Let $G$ be an ordered graph with $n$ vertices and $m$ 
edges, $\eps\in \ZZ^m$, and an orientation $\sigma$ of $G$ be fixed.  Let $S^\sigma$ 
be the $\RR$-cycle space of $G^\sigma$.  We define 
\[
    \hat{S}^\sigma := \left\{\begin{pmatrix}
        \v \\ \trans{\eps} \v
    \end{pmatrix} : \v\in S^\sigma\right\}
\]
The subspace $\hat{S}^\sigma\subseteq \RR^{m+1}$ is the image of $S^\sigma$ under 
an injective linear transformation.  
\end{definition}

\begin{lemma}\label{lem: c med}
Let $G$ be a connected ordered graph with $n$ vertices and $m$ 
edges, $\sigma$ an orientation of $G$, $\p$ a configuration of $n$ points, and 
$\eps\in \{-1,0,1\}^m$ an error vector.  Then $\hat{S}^\sigma$ is contained 
in the linear span of a set of vectors in 
$\Lat(\l)$, 
each of norm at most $\sqrt{2m}$. 
\end{lemma}
\begin{proof}
Let $\w_\rho$ be the signed cycle vector of a cycle 
$\rho$ in the oriented graph $G^{\sigma_\p}$.  
Define $e_\rho = \trans{\w_\rho}\l$.
From Lemma 
\ref{lem:stress}, we have 
\[
    e_\rho = \trans{\w_\rho}\l = \trans{\w_\rho}(\ell(\p) + \eps) = \trans{\w_\rho}\eps
\]
We then define $\hat{\w}_\rho := \begin{pmatrix}
    \w_\rho \\ e_\rho
\end{pmatrix}$. 
For any $\r\in \ZZ^m$, the column span of the lattice generating matrix $L(\r)$ 
is exactly vectors of the form 
\bna
\label{eq:vecInLat}
    \hat{\x} = \begin{pmatrix}
    \x \\ \trans{\r} \x
\end{pmatrix}.
\ena
The vector $\hat{\w}_\rho$ is in $\Lat(\l)$, because it is integral and is in the 
column span of $L(\l)$. 

From Proposition \ref{prop: matroid basics}, there is a basis $\w_{\rho_1}, \ldots, \w_{\rho_c}$
of signed cycle vectors for $S^{\sigma_\p}$, where $c:=m-n+1$.
As the images under the injective linear map 
of Definition~\ref{def:hatS},
the vectors $\hat{\w}_{\rho_1}, \ldots, \hat{\w}_{\rho_c}$ 
are linearly independent, and, by the discussion above, in $\Lat(\l)$.  Because the 
vectors $\w_{\rho_i}$ and $\eps$ are $\pm 1$ vectors, we get, for each $1\le i\le c$, 
the estimate
\[\|\hat{\w}_\rho\|^2_2 = \|\w_\rho\|^2_2 + |e_\rho|^2 \le 2m\]
It now follows that $\Lat(\l)$ contains $c$ linearly independent vectors of norm 
at most $\sqrt{2m}$ that are contained in $\hat{S}^\sigma$.  As $\hat{S}^\sigma$
is $c$-dimensional, these must be a spanning set. 
\end{proof}
Now we proceed with the probabilistic analysis of our algorithm.  The 
idea, which is already present in Frieze's work \cite{frieze}, 
is that if $\l$ (instead of $\p$) has been selected uniformly at random, 
then $\Lat(\l)$ would, WHP, not contain  any vector of 
length $\sqrt{2m}2^{m/2}$.  The point of Lemma \ref{lem: c med} is that, 
because $\l$ is a (corrupted) measurement  of a configuration $\p$, 
$\Lat(\l)$ must contain some ``planted'' short vectors of norm at most 
$\sqrt{2m}$ that span the linear space $\hat{S}^\sigma$.  We won't be able to 
recover these efficiently, so the next step of the analysis 
is to show that, WHP 
(in $\p$), every vector of norm at most $\sqrt{2m}2^{m/2}$ in 
$\Lat(\l)$ lies in $\hat{S}^\sigma$.  Informally, $\Lat(\l)$ 
is still ``random enough'' looking that we can spot the vectors in 
it that arise from cycle relations.
\begin{definition}
We call any vector of norm at most $\sqrt{2m}2^{m/2}$ \defn{medium sized}.
\end{definition}
We defer the proof of the following key estimate to 
Section \ref{sec:pp}.
\begin{prop}
\label{prop:prob}
Let $G$ be a connected graph with $m$ edges and $n$ vertices,  $\eps\in \{-1,0,1\}^m$, 
and $\hat{\x}\in \ZZ^{m+1}$.  If $\p$ is selected uniformly 
at random from $[B]^n$, then the probability that 
$\hat{\x}$ is in $\Lat(\ell(\p)+\eps)$ 
but $\hat{\x} \not\in 
\hat{S}^{\sigma_\p}$ is at most $\frac{2^n n!}{B}$.
\end{prop}
Using the key estimate, we prove the result mentioned in the 
informal discussion above.
\begin{lemma}
\label{lem:isp}
Let $G$ be a connected graph with $m$ edges and $n$ vertices.  
If $\p$ is selected uniformly at random from $[B]^n$,
then, with probability at least 
\[
    1 - \frac
{2^{m^2/2} \;\; 2^{O(m\log m)}}
{B}
\]
for every $\eps\in \{-1,0,1\}^m$, every 
medium sized vector in $\Lat(\ell(\p) + \eps)$
is in $\hat{S}^{\sigma_\p}$.
\end{lemma}
\begin{proof}
For the moment, fix $\eps \in \{-1,0,1\}$.  Set $K=(2m)^{1/2}2^{m/2}$.
If $\hat{\x}\in \ZZ^{m+1}$ is a medium sized integer vector, then 
certainly all its coordinates have magnitude at most 
$K$.  Hence, there are at most $(2K+1)^{m+1} \le (3K)^{m+1}$ medium sized 
integer vectors.  We bound this latter quantity as follows
\[
    \left(
        3(2m)^{1/2}2^{m/2}
    \right)^{m+1}
    \le 
    2^{\frac{m(m+1)}{2}}
    2^{\frac{m+1}{2}}
        3^{(m+1)}
        m^{\frac{m+1}{2}}
     = 2^{m^2/2}2^{O(m\log m)}
\]
With $\eps$ fixed, Proposition \ref{prop:prob} then implies that, if $\p$ is selected 
uniformly at random from $[B]^n$, the probability that there is some medium 
sized $\hat{\x}$ so that $\hat{\x}\in \Lat(\ell(\p) + \eps)$ and $\hat{\x}\notin \hat{S}^{\sigma_\p}$
is at most 
\[
    \frac{2^n n!}{B}2^{m^2/2}2^{O(m\log m)}
\]
Finally, we let $\eps$ vary over $3^m$ choices to conclude that the statement 
fails with probability at most 
\[
3^m\frac{2^n n!}{B}2^{m^2/2}2^{O(m\log m)}
\]
Because $G$ is connected it has at least $n-1$ edges, so we can absorb the 
term $3^m2^nn!$ into the $2^{O(m\log m)}$.  The final failure probability is 
\[
\frac{2^{m^2/2} \;\; 2^{O(m\log m)}}
{B}
\]
as claimed.
\end{proof}
Now picking $B$ to make the probability of success high is 
routine.
\begin{lemma}\label{lem: failprob}
Let $b\ge(\frac{1}{2}+\delta)m^2$ for fixed $\delta > 0$.
For $m$ sufficiently large (depending on $\delta$), 
if $\p$ is selected uniformly at random from $[2^b]^n$, 
with probability at least $1 - \frac{1}{2^{\delta m^2/2}}$,
for every $\eps \in \{-1,0,1\}^n$, if $\hat{\x}$ is medium 
sized and $\hat{\x} \in \Lat(\ell(\p)+\eps)$, then 
$\hat{\x}\in \hat{S}^{\sigma_\p}$.
\end{lemma}
\begin{proof}
For sufficiently large $m$, the quantity $2^{m^2/2}2^{O(m\log m)}$
is bounded by $2^{(\frac{1}{2} + \frac{\delta}{2})m^2}$.  The 
probability of the event in the statement fails is, by Lemma~\ref{lem:isp},
is then at most 
\[
    \frac{2^{(\frac{1}{2} + \frac{\delta}{2})m^2}}{2^{(\frac{1}{2}+\delta)m^2}} = \frac{1}{2^{\delta m^2 /2}}
\]
\end{proof}
At this point, we can analyze the algorithm.
\begin{prop}\label{prop:step1}
For all $\delta > 0$, there is an $M\in \NN$
and a polynomial time algorithm that correctly computes,
with probability at least $1 - 2^{-\delta m^2/2}$,
for any connected ordered graph
$G$ with $n$ vertices and $m\ge M$ edges, and any 
noise vector $\eps\in \{-1,0,1\}^m$,
a basis for the space
$S^{\sigma_p}$ from $\l = \ell(\p) + \eps$.
The probability is with respect to $\p$ chosen uniformly at 
random among configurations of $n$ integer points in
$[B]$, where $B = 2^b$ and $b \ge (\frac{1}{2} + \delta)m^2$. 
\end{prop}
\begin{proof}
Let us recall that, given $\l$, the algorithm 
first computes an LLL reduced basis $\Lambda$,
of $\Lat(\l)$ and then returns the truncation 
of every medium sized vector in $\Lambda$.

By Lemma \ref{lem: c med}, $\Lat(\l)$ contains 
at least $c$ linearly independent vectors of 
norm at most $\sqrt{2m}$.  As these vectors can be 
extended to a lattice basis, the approximation guarantee of an 
LLL reduced basis \cite[Proposition 1.12]{LLL} implies that 
$\Lambda$ must contain 
at least $c$ vectors of norm at most $\sqrt{2m}2^{(m-1)/2}$.
Let us call these $\hat{\v}_1, \ldots, \hat{\v}_c$.

Since  $b \ge (\frac{1}{2} + \delta)m^2$, 
by Lemma \ref{lem: failprob}, there is an $M > 0$ so that, if $m > M$, 
with probability 
at least $1 - 2^{-\delta m^2/2}$, \emph{all} the vectors 
$\hat{\v}_i$ are in the $c$-dimensional linear space $\hat{S}^{\sigma_\p}$.
As the $\hat{\v}_i$ are linearly independent, they are a basis for 
$\hat{S}^{\sigma_\p}$.
The algorithm returns the truncated vectors  $\v_1, \ldots, \v_c$.
By construction, the $\v_i$ are in $S^{\sigma_\p}$.  They 
must also be linearly independent. Otherwise the 
$\hat{\v}_i$ would be linearly dependent as the image of a dependent set under the linear transformation
of Definition~\ref{def:hatS}.  Hence, the algorithm 
returns a basis for $S^{\sigma_\p}$ with 
probability at least $1 - 2^{-\delta m^2/2}$.

The LLL algorithm runs in polynomial time \cite{LLL},
and all the other steps are evidently polynomial, so the 
algorithm to compute a basis of $S^{\sigma_\p}$ runs in 
polynomial time.
\end{proof}
We note, for later, than the polynomial running time of the 
algorithm to find a basis for $S^{\sigma_p}$ implies 
that the coordinates of the vectors it returns have a number of bits 
that is polynomially bounded in $m$.

\subsection{More Error/Fewer Bits}
In the following remarks, we touch on a number of noise related issues.

\begin{remark}
\label{rem:err3}
Lemma~\ref{lem: failprob}, Proposition~\ref{prop:step1},
and our main Theorems~\ref{thm:main1} and~\ref{thm:main2}
can be generalized somewhat.
In particular we can set 
$b\ge(\frac{1}{2}+\delta)m^2-f(m)$ where $f(m)=o(m^{2})$.
In this case, the result can still be shown to hold for sufficiently large 
values of $m$ (which now depends on $f$ as well as $\delta$).
\end{remark}

\begin{remark}
\label{rem:err1}
Suppose that instead of errors in $\{-1,0,1\}$, 
we had integer errors in 
$\{-(2^k-1),...,2^k-1\}$, for some $k$.
The estimate on the size of the small vectors 
in Lemma \ref{lem: c med} increases somewhat 
to $\sqrt{2(2^k - 1)m}$.  This means, that in 
the proof of Lemma \ref{lem:isp}, we need to increase the 
quantity $K$ to $\sqrt{2(2^k - 1)m}2^{m/2}$ and replace the 
$3^m$ choices of $\eps$ with $(2^{k+1}-1)^m$.  This second 
change raises the failure probability in the conclusion 
of Lemma \ref{lem:isp} by a factor of $(\frac{2^{k+1}-1}{3})^m$.

Alternatively, by dividing by $2^{k+1}$, we could model the resulting rational numbers as a real valued input, as in 
Section~\ref{sec:real}, with errors bounded in magnitude by $1/2$. Then we can round the data, putting us 
back to the $\{-1,0,1\}$ error model, but on integers in the interval  
$\{1,\ldots,2^{b-k-1}\}$.
Then in the proof of of Lemma~\ref{lem:isp}, the $B=2^b$ term in the 
denominator would become $2^{b-k-1}$. This would raise the probability of failure instead by only $2^{k+1}$.

Since $(\frac{2^{k+1}-1}{3})^m$ is bigger than 
$2^{k+1}$, 
we see, that if we have 
integer errors in 
$\{-(2^k-1),...,2^k-1\}$, for some $k$, 
we are much better off
rounding off the lowest $k+1$ bits and treating this as 
 $\{-1,0,1\}$ error, than keeping around the noise! This observation is in line with the 
 conclusions of~\cite{gam2}.
\end{remark}

\begin{remark}
\label{rem:err2}
If we do have  errors of
magnitude bounded by some constant $K$, 
the proof of Lemma~\ref{lem:isp} can be modified to accommodate the larger error.
Likewise Lemma~\ref{lem: failprob} and its proof still hold  but the sufficiently large values 
required for $m$ will need to increase.
Indeed, as long as 
$\log(K)=o(m)$
Lemma~\ref{lem: failprob} can still be shown to hold.
For this generalization, the statement of Lemma~\ref{lem:isp} 
needs to have the $O(m\log m)$ term replaced by $o(m^{2})$.
This is in line with~\cite[Theorem 2.1]{gam2}, where the 
role of error is tracked more carefully than we are doing here.
\end{remark}

\subsubsection{Proof of Proposition \ref{prop:prob}}
\label{sec:pp}
Now we develop the proof of Proposition~\ref{prop:prob}.  In 
this section, we fix a connected ordered graph $G$ with 
$n$ vertices and $m$ edges, as well as an error 
vector $\eps \in \{-1,0,1\}^m$ and any vector 
$\hat{\x}\in \ZZ^{m+1}$.  Let $B\in \NN$ be given, and 
define 
\[
    \cS(\hat{\x},\eps) = \{\p \in [B]^n : 
    \text{$\hat{\x}\in \Lat(\ell(\p) + \eps)$ and 
    $\hat{\x}\notin \hat{S}^{\sigma_\p}$} \}
\]
The Proposition asserts that 
\bna 
\label{eq: main estimate}
    |\cS(\hat{\x},\eps)| \le 2^n n! B^{n-1}
\ena
Establishing  \eqref{eq: main estimate} is 
complicated by the fact that the measurement map 
$\ell$ is non-linear.  To deal with this, we first 
define, for each orientation $\sigma$ of $G$, the 
set 
\[
    \cS(\hat{\x},\eps,\sigma) = 
    \{ \p \in \cS(\hat{\x},\eps) : \sigma_\p = \sigma \}
\]
And we have
\bna 
\label{eq: sad} 
 |\cS(\hat{\x},\eps)| = \sum_{\sigma} |\cS(\hat{\x},\eps,\sigma)|
\ena
where the sum ranges over all 
vertex consistent orientations of $G$. By Lemma 
\ref{lem: point ordering count}, there are at most 
$n!$ vertex consistent orientations.

Next we fix a $\sigma$ so that $\cS(\hat{\x},\eps,\sigma)$
is non-empty.  For any $\p$ in $\cS(\hat{\x},\eps,\sigma)$, 
we know that $\ell(\p) = \ell_{\sigma}(\p)$.  
The map  
$\ell_\sigma(\p)$ is linear.
Since $G$ is connected, Proposition \ref{prop: matroid basics}
implies that the rank of $\ell_{\sigma}$ is $n-1$, so it 
has a one-dimensional kernel.  It follows that the 
fiber $\ell^{-1}_\sigma(\ell_\sigma(\p))$ 
consists only of translations of $\p$.  
We observe that there are at most 
$B$ of these where the points $p_i$ remain in the range $[B]$.

Meanwhile, because $\p\in \cS(\hat{\x},\eps,\sigma)$, 
the measurement $\ell(\p) + \eps$ must be in the set 
\[
    \cB(\hat{\x},\sigma) = 
    \{ \r\in \{-1, \ldots, B\}^m : 
        \text{$\hat{\x}\in \Lat(\r)$, 
        $\hat{\x}\notin \hat{S}^{\sigma}$,
        and $\hat{S}^\sigma$ is in the 
        linear span of $\Lat(\r)$ }
    \}
\]
where we have used Lemma \ref{lem: c med} to ensure that 
$\Lat(\ell(\p) + \eps)$ contains a basis for $\hat{S}^\sigma$.
From the observations above, we obtain 
\bna 
\label{eq: bad sad}
    |\cS(\hat{\x},\eps,\sigma)| \le B |\cB(\hat{\x},\sigma)|
\ena

The next task is to bound the size of a
non-empty $\cB(\hat{\x},\sigma)$.
We do this by showing that 
such a $\cB(\hat{\x},\sigma)$ is 
contained in an $(n-2)$-dimensional affine subspace 
of $\RR^m$.  
Our main tool will be the following simple fact.
\begin{lemma}
\label{lem:aff}
Let $A$ be a
$d$-dimensional affine subset of $\RR^{m}$. 
Then 
\[
    |\{-1,0,\ldots, B\}^m \cap A| \le (B + 2)^d \le (2B)^d
\]
\end{lemma}
\begin{proof}
For convenience, set $C^m = \{-1,0,\ldots, B\}^m$ and define 
$C^d =  \{-1,0,\ldots, B\}^d$.
Without  loss of generality, 
the projection of $A$ onto
the first $d$ coordinates of $\RR^m$ by forgetting the last $m-d$ coordinates 
is  bijective (if not, one can pick a different coordinate subspace), 
and in particular injective. This
affine map  sends  points in $A\cap C^m$ to 
points in $C^d$.
It follows that $|A\cap C^m| \le |C^d| = (B+2)^d$.  
Since $B\ge 2$, then $B + 2 \le 2B$.
\end{proof}

Returning to the proof of the Proposition, we let
$\{\hat{\v}_1,\ldots, \hat{\v}_c\}$ be some fixed basis for 
$\hat{S}^{\sigma}$; recall that, because $G$ is connected 
the dimension of $\hat{S}^\sigma$ is $c = m - n + 1$. 
Each of these basis vectors is of the form 
 $\hat{\v}_i =: \begin{pmatrix}
    \v_i \\ f_i
\end{pmatrix}$ for
some $\v_i$ and $f_i$.
Let us record for later the following linear system
with variable $\r \in \RR^m$:
\bna
\label{eq:LinW}
\begin{pmatrix}
\trans{\v}_1 \\ \vdots \\ \trans{\v}_c
\end{pmatrix}
\r =
\begin{pmatrix}
f_1
\\ \vdots \\ 
f_c 
\end{pmatrix}
\ena
Since the $\v_i$ are linearly independent, 
the rank of this system is $c$. 
(For linear independence recall that
if the $\v_i$ were dependent, then the 
$\hat{\v}_i$ would be have to be
linearly dependent as the image of a dependent set under the linear transformation
of Definition~\ref{def:hatS}.)

Now suppose that for some fixed $\r\in \RR^m$, 
$\r\in \cB(\hat{\x},\sigma)$.  This implies that 
$\Lat(\r)$ contains a basis $\hat{\w}_1, \ldots, 
\hat{\w}_c$ of $\hat{S}^\sigma$.  
Each of these basis vectors 
is of the form 
$\hat{\w}_i = \begin{pmatrix}
    \w_i \\ e_i
\end{pmatrix}$, and as 
each $\hat{\w}_i$ is in $\Lat(\r)$,
$\r$ satisfies (recalling Equation (\ref{eq:vecInLat})) the following linear system:
\bna
\label{eq:eqLin}
\begin{pmatrix}
\trans{\w}_1 \\ \vdots \\ \trans{\w}_c
\end{pmatrix}
\r =
\begin{pmatrix}
e_1 \\ \vdots \\ e_c
\end{pmatrix}
\ena
Meanwhile, putting the $\hat{\v}_i$ and $\hat{\w}_j$
as the columns of $(m+1)\times c$ matrices, we can find an 
invertible $c\times c$ matrix $W$ so that 
\[
\begin{pmatrix}
    \hat{\v}_1 & \cdots & \hat{\v}_c
\end{pmatrix} = 
\begin{pmatrix}
    \hat{\w}_1 & \cdots & \hat{\w}_c
\end{pmatrix} W
\]
it now follows that  
$\r$ must also satisfy this next linear system, which 
which is equivalent to \eqref{eq:eqLin}
\[
\trans{W}
\begin{pmatrix}
\trans{\w}_1 \\ \vdots \\ \trans{\w}_c
\end{pmatrix}
\r =
\trans{W}\begin{pmatrix}
e_1 \\ \vdots \\ e_c
\end{pmatrix}
\]
But this system is exactly  \eqref{eq:LinW}. 
So we see that $\r$ satisfies \eqref{eq:LinW}
whether or not the $\hat{\v}_i$ are in $\Lat(\r)$.

Up to this point, we have not yet used the other hypothesis 
coming from our assumption that $\r\in \cB(\hat{\x},\sigma)$:
that $\hat{\x}\in \Lat(\r)$ and $\hat{\x}\notin \hat{S}^\sigma$.
Now we will.  Let $\hat{\x} = \begin{pmatrix}
    \x \\ g
\end{pmatrix}$.  Because $\hat{\x}\in \Lat(\r)$, 
$\r$ satisfies the linear equation 
\[
    \trans{\x}\r = g
\]
and, because 
$\cB(\hat{\x},\sigma)$ is not empty, we have
$\hat{\x}\notin \hat{S}^\sigma$, so the 
rank of the the linear system that adds this equation 
to \eqref{eq:LinW}, namely 
\bna
\label{eq:eqLin2}
\begin{pmatrix}
\trans{\v}_1 \\ \vdots \\ \trans{\v}_c\\ \trans{\x}
\end{pmatrix}
\r =
\begin{pmatrix}
f_1 \\ \vdots \\ f_c \\g
\end{pmatrix}
\ena
must be $c+1$.  Because we assume that $\cB(\hat{\x},\sigma)$ is non-empty,
this final system \eqref{eq:eqLin2} is 
consistent, and has, as its solution space, an  
affine subspace $A$ of dimension $m - (c - 1) = n -2$.
Since \eqref{eq:eqLin2} depends only on the data 
defining $\cB(\hat{\x},\sigma)$, we conclude that 
\[
    \cB(\hat{\x},\sigma)\subseteq A \cap \{-1, \ldots, B\}^m
\]
And now using Lemma~\ref{lem:aff} we get 
\bna 
\label{eq: bad}
|\cB(\hat{\x},\sigma)| \le (2B)^{n-2}
\ena
Finally, plugging \eqref{eq: bad} into \eqref{eq: bad sad}, 
we get 
\[
    |\cS(\hat{\x},\eps,\sigma)| \le 2^n B^{n-1}
\]
whenever the l.h.s. is non-zero.  Plugging this 
estimate into \eqref{eq: sad}, and summing over the 
at most $n!$ non-zero terms, we obtain \eqref{eq: main estimate}, 
and the proposition is proved.
$\hfill\qed$

\subsection{Reconstruct $G$}

Next we show how to reconstruct $G$ from
$S^{\sigma_\p}$. Our main tool will be an algorithm
by Seymour~\cite{sey}.
First we recall some standard facts about graphic matroids.
A standard reference for matroids is \cite{oxley}; the material
in this section can be found in \cite[Chapter 5]{oxley}.

\begin{definition}\label{def: graphic matroid}
Let $G = (V,E)$ be an ordered graph.  The \defn{graphic matroid} of 
$G$ is the matroid on the ground set $E$ that has, 
as its independent sets, the subsets of $E$ corresponding 
to acyclic subgraphs of $G$.
\end{definition}
If $G$ is connected and has $n$ vertices, the rank of its
graphic matroid is $n-1$.
The following lemma is due to Whitney \cite{whitney-matroid}.  
\begin{lemma}\label{lem:ind oracle1}
Let $G^\sigma$ be any orientation of an
ordered graph $G = (V,E)$.  Then 
a subset $E'\subseteq E$ 
is independent in the graphic matroid 
of $G$ if and only if $S^\sigma\subseteq \RR^m$ contains no 
non-zero vector with support contained in $E'$. 
\end{lemma}

\begin{definition}\label{def: ind oracle}
Let $\cM$ be a matroid defined on a ground set $E$.  An 
\defn{independence oracle} for $\cM$ takes as its input
a subset $E'$ of $E$ and outputs ``yes'' or ``no'' depending on 
whether $E'$ is independent in $\cM$.
\end{definition}
Since $G$ has $n$ vertices and is connected, the 
dimension of $S^\sigma$ is $m - n + 1$.  Given 
an $m\times (m - n + 1)$ matrix $A$ that has as its 
columns a basis of $S^\sigma$, we can check whether 
a set $E'\subseteq E$ is independent follows.  Let $A'$ be the 
matrix obtained from $A$ by discarding the rows 
corresponding to $E'$.  If $A'$ has a non-zero 
vector $\v$ in its null space then $A\v$ 
has zeros in coordinates corresponding to 
edges outside of $E'$.  Since $\v$ is non-zero and 
$A$ has independent columns, $A\v$ is non-zero, 
so, by Lemma \ref{lem:ind oracle1}, $E'$ is 
dependent.  On the other hand, if $S^\sigma$ 
contains a non-zero vector $\w$ with support in $E'$, 
there is a $\v$ so that $\w = A\v$, and this $\v$ must be in the 
null space of $A'$. 

It now follows that, if every entry of $A$ has a 
number of bits polynomially bounded in $n$ 
and $m$, we can check independence of $E'$
in polynomial time by forming $A'$ and computing its 
rank with Gaussian elimination.  This gives us a 
polynomial time independence oracle for 
the graphic matroid of $G$.

A striking result of Seymour \cite{sey}, which builds on 
work of Tutte \cite{tutte} (see also \cite{bixby}) is that
an independence oracle can be used to efficiently 
determine a graph $G$ from its matroid.
\begin{theorem}[\cite{sey}]\label{thm: seymour}
Let $\cM$ be a matroid over an ordered ground set.  There is an 
algorithm that uses an independence oracle to determine 
whether $\cM$ is the graphic matroid of some ordered graph $G$.  If
so, the algorithm outputs some such $G$.
If the independence oracle is polynomial time, so 
is the whole algorithm.
\end{theorem}
Since the matroid of an ordered
graph $G$ depends only 
on whether each subset of edges is a cycle or 
not, if $G$ and $H$ are isomorphic as ordered 
graphs, they will have the same matroid and so 
Seymour's algorithm can only return an ordered 
graph $H$ that is isomorphic to $G$.
A foundational 
result of Whitney \cite{whitney} says that if an 
(ordered)
graph $G$ is $3$-connected, then $G$ is determined 
up to (ordered) graph isomorphism by its graphic matroid.  

\begin{corollary}\label{cor: seymour}
If $G$ is $3$-connected, then Seymour's algorithm 
outputs an ordered graph isomorphic to $G$, given an 
independence oracle for the graphic matroid of $G$.  
\end{corollary}

Which for our purposes, together 
with Lemma~\ref{lem:ind oracle1}
gives us:

\begin{prop}\label{prop:step2}
Let $G$ be a $3$-connected ordered graph with $n$ vertices and $m$ edges. Let $\sigma$ be an orientation and $A$ be 
a $m\times c$ matrix with column span 
$S^\sigma$. Suppose that every  entry in $A$ has a bit
complexity that is polynomial in $n$ and $m$.
Then there is an algorithm,
with running time polynomial in $n$ and $m$
that 
correctly computes 
$G$ from $A$, up to isomorphism.
\end{prop}

\subsection{Reconstructing $\sigma_\p$}\label{sec: reconstruct sigma}

Next we show how to reconstruct $\sigma_\p$ from
$G$ and 
$S^{\sigma_\p}$. 

\begin{lemma}\label{lem:ind oracle2}
Let $G^\sigma$ be any orientation of an
ordered graph $G = (V,E)$.  
Let $E'$ be the edge set of a cycle $\rho$ in $G$.
Then any non-zero vector in $S^\sigma$ with 
support contained in coordinates corresponding to $E'$ is a scale of the 
signed cycle vector $\w_{\rho,\sigma}$.
\end{lemma}
\begin{proof}
As a cycle, every proper subset of $E'$ 
must be independent in the graphic matroid of $G$.  Hence, by Lemma \ref{lem:ind oracle1}, 
any non-zero vector in $S^\sigma$ that is supported on coordinates corresponding to 
$E'$ must have no zero entries.  If there are two such vectors that 
are not scalar multiples of each other, then an appropriate linear 
combination would be non-zero and have smaller support.  The resulting 
contradiction says that, up to scale there is a single vector 
in $S^\sigma$ supported on $E'$; since $\w_{\rho, \sigma}$ is such a 
vector, we are done.
\end{proof}
If we have an $m\times (m - n + 1)$ matrix $A$ 
with column span $S^\sigma$, we can find the 
vector $\pm\w_{\rho, \sigma}$ from the edge set $E'$ 
of $\rho$ by finding a non-zero vector $\v$ in the null 
space of the matrix $A'$ obtained by discarding the 
rows of $A$ corresponding to $E'$.  We then have 
$\alpha\w_{\rho, \sigma} = A\v$, for some scalar $\alpha\in \RR$.  Lemma \ref{lem:ind oracle2}
implies that $\v$ is unique up to scale, and then we 
can scale $A\v$ to obtain a vector with only entries $\pm 1$, 
which must be $\pm \w_{\rho, \sigma}$.

In the algorithm to recover $\sigma_\p$, we will 
use the following graph-theoretic notion.
\begin{definition}\label{def: ear decomposition}
Let $G$ be a graph.  An \emph{ear decomposition} of $G$  
is a sequence of cycles 
\[
    \rho_1, \ldots, \rho_k
\]
such that, for each $2\le j\le k$
\[
    \rho_j \setminus (\rho_1 \cup \cdots \cup \rho_{j-1})
\]
is a path, and the edge set of $G$ is the union of 
the cycles $\rho_i$.
\end{definition}
A graph is $2$-connected, if and only if it has an ear decomposition  
\cite[Proposition 3.1.1]{diestel}, and an ear decomposition can be computed in polynomial 
time (e.g., \cite{lovasz-ear}).

\begin{prop}\label{prop:step3}
Let $G = (V,E)$ be a $2$-connected ordered graph, and
$A$ an $m\times (m - n + 1)$ matrix whose  column span,  $S^\sigma$,
is the $\RR$-cycle space of 
an unknown orientation of $G$. 
Also assume that the entries of $A$
have a number of bits that is bounded polynomially in $m$.
Then  we can find $\sigma$
from this data in polynomial time,
up to a global flipping of the entire 
orientation.
\end{prop}
\begin{proof}
Lemma \ref{lem:ind oracle2} allows us to construct, an orientation 
of any cycle in $G$ that is consistent with $\sigma$ or 
the orientation $-\sigma$, in which all the edges are 
reversed.  To reconstruct $\sigma$ (or $-\sigma$), 
we will iteratively orient the whole graph, working 
one cycle at a time, in a way that makes sure that all the 
cycles are oriented consistently with each other, and thus globally correct up to a 
single sign.

We first compute an ear decomposition 
\[
    \rho_1, \ldots, \rho_k
\]
Using the matrix $A$ and Lemma \ref{lem:ind oracle2}, we find a signed 
cycle vector $\w_1 = \pm\w_{\rho_1,\sigma}$.  
Now, for each $j = 2, \ldots, k$, we use $A$ and 
Lemma \ref{lem:ind oracle2} to find a vector 
$\w_j = \pm \w_{\rho_j,\sigma}$ so that the signs 
in coordinates corresponding to  
edges in $\rho_j\cap (\cup_{i=1}^{j-1} \rho_i)$ 
agree with the signs in $\w_1, \ldots, \w_{j-1}$.  
The defining property of an ear decompostion is that 
\[
    \rho_j\setminus (\cup_{i=1}^{j-1} \rho_i)
\]
is a path.  Because $\rho_j$ is a cycle, this implies that 
\[
    \rho_j\cap (\cup_{i=1}^{j-1} \rho_i)
\]
is non-empty.  Hence, exactly one of the two possible 
orientations of $\rho_j$ consistent with $\sigma$ or $-\sigma$ 
will also be consistent with the previous choices of the 
algorithm.
At the end of the process, the coordinate corresponding to 
every edge is either $\pm 1$ in at least one of the vectors 
$\w_j$, and the signs will be the same whenever non-zero.  
Hence, we may read off the orientation $\sigma$, up to 
an unknown sign.  

Each of the steps is polynomial time.  As discussed above, 
an ear decomposition can be computed in polynomial time.
The number of cycles in the ear decomposition is at 
most $m$, so there are at most $m$ calls to the polynomial 
time independence oracle obtained from $A$.  All of the other 
steps are evidently polynomial, and so the algorithm is too.
\end{proof}

Taken together, Propositions~\ref{prop:step1}
(with $\delta$ set to $1/2$),
\ref{prop:step2},
and~\ref{prop:step3}, 
constitute a proof of our 
Theorem~\ref{thm:main2}.

\subsection{Reconstruct $\p$}
Now armed with $\l$, $G$  and $\sigma_\p$
(up to a global flip)
we can reconstruct $\p$, up to translation and
reflection.

In the noiseless setting, we can simply
traverse a spanning tree of $G$ and lay
out each of the vertices in a greedy manner.
This step completes the 
proof of Theorem~\ref{thm:main1}.

In the setting with noise, we can set
up a least-squares problem, finding
the $\p$ that minimizes the squared error
in the length measurements, as follows.
Suppose 
$G$ and $\sigma_\p$ have been successfully reconstructed.
Without loss of generality,
assume that we 
have chosen the global sign
of $\sigma_\p$ correctly as well.
Then we can set up the linear system
\[
    M \x = \l,
\]
where $M$ is  the signed edge incidence matrix
 $M^{\sigma_\p}(G)$.

The matrix $M$ has a kernel spanned by
the ${\bf 1}$ (all ones)
vector, which we now want
to mod out. To do this, let $Q$ be 
an $n$-by-$(n-1)$ matrix whose columns form an 
orthonormal basis for ${\bf 1}^\perp$.
We now consider the linear system
\[
    M Q\y = \l    
\]
To solve for the least squares solution, we set up
the normal equations
\[
  \trans{Q} \trans{M} M Q\y = \trans{Q} \trans{M} \l    
\]
with solution
\[
   \y^* = (\trans{Q} \trans{M} M Q)^{-1} (\trans{Q} \trans{M}) \l    
\]
We then output the configuration $\p^*:=Q \y^*$.

We now want to compare this output to 
$\q:=Q\trans{Q}\p$, which is the configuration $\p$ translated to have
its center of mass at the origin.
Thus we want to bound 
\[||\p^*-\q||=||Q\y^*-Q\trans{Q}\p||=||\y^*-\trans{Q}\p||.\]
Let us define the correct lengths as
$\l' := M \p = (M Q)(\trans{Q}\p)$ and the residual
$\r := \l - \l'$.
Then 
\ba
   \y^* &=& (\trans{Q} \trans{M} M Q)^{-1} (\trans{Q} \trans{M}) [\l' + \r] \\
&=& \trans{Q} \p +
 (\trans{Q} \trans{M} M Q)^{-1} (\trans{Q} \trans{M}) \r 
\ea
and, hence, 
\[
\y^* - \trans{Q}\p = 
(\trans{Q} \trans{M} M Q)^{-1} (\trans{Q} \trans{M}) \r .
\]
Under our noise assumption, 
the norm of $\r$ is at most
$\sqrt{m}$.

Recognizing $\trans{M}M$ as the unweighted Laplacian matrix 
of the graph $G$, we see 
its smallest non-zero eigenvalue
is at least $\frac{1}{n^2}$ \cite[Equation (6.10)]{mohar} and
the largest singular value of $M$
is  at most
$\sqrt{n}$ \cite[Theorem 2.2(c)]{mohar}.  
The only effect of $Q$ on the spectral
decomposition of $\trans{M}M$ is  to
remove the $0$ eigenvalue. 
Likewise, its only effect on $M$
is to remove the $0$ singular value.
Putting 
things together, we see that, up to translation and 
reflection on $\p$, the norm of 
$ \p^*-\q$ is at most $n^{2.5}\sqrt{m}$.
This completes the proof of
Theorem~\ref{thm:main2b}.

\subsection{Labeled Setting}
In the labeled setting the algorithm is even easier.
In particular, we can  skip the step of reconstructing $G$
from $S^{\sigma_\p}$. (That is the only step that depends on
$3$-connectivity.)

Alternatively, since we have access to $G$,  we can also
apply a more expensive, but conceptually more simple algorithm.
Instead of using LLL as in Section \ref{sec:lll} to 
recover $S^{\sigma_\p}$ 
all at once, 
we can compute an ear decomposition of $G$ 
(as described in Section \ref{sec: reconstruct sigma}), 
and then extract the orientation of the edges by 
solving a single SUBSET-SUM problem (using LLL)
separately
for each cycle. 
(These are not independent SUBSET-SUM problems, but 
there are at most $n$ of them, so the failure 
probability is still exponentially small in $m$.)
Then we continue with the computation as described in
Section~\ref{sec: reconstruct sigma}.

\section{Detecting failures}
\label{sec:fail}
Theorem \ref{thm:main1} provides only a high probability 
guarantee of reconstructing $G$, in the unlabeled 
setting, and $\p$, in the labelled and unlabelled 
settings.  We now discuss where the algorithm can fail, and 
some tests to detected failure.

The first observation is that the only step requiring a 
probabilistic analysis is the computation of $S^{\sigma_\p}$
using an LLL reduced basis for $\Lat(\l)$.  After this, 
the rest of the steps are deterministic and, if $S^{\sigma_\p}$
is correctly computed, will compute $G$ and then $\p$. 
Hence, every failure of 
the algorithm arises from failure of the LLL step 
to correctly compute $S^{\sigma_\p}$.  Moreover, 
since the truncations of the medium sized vectors in $\Lambda$ 
span a subspace $S$ of $\RR^m$ that contains $S^{\sigma_\p}$, 
failure occurs iff $S\supsetneq S^{\sigma_\p}$.

One example of how this can happen is when the 
pair $(G,\p)$ is not ``globally rigid''.  This means 
there is another configuration $\q$ so that $\l = \ell(\p) = \ell(\q)$,
but $\q$ is not related to $\p$ by translation or reflection.  
Since we assume that $G$ is $3$-connected, there is some 
cycle in $G$ that is oriented differently in $\sigma_\p$ 
and $\sigma_\q$.  Since $\ell$ satisfies the cycle relations
from both $G^{\sigma_\p}$ and $G^{\sigma_q}$, the lattice 
reduction will find more than $c$ medium-sized relations.
We don't know that failures of global rigidity are the 
only way that the lattice reduction step can fail.

\subsection{Labeled setting} 
In labeled problems, the algorithm has access to $n$, and 
thus also $c$.  Since the computation of $S^{\sigma_\p}$
succeeds if and only if $\Lambda$ has exactly $c$ medium 
sized vectors in it, the algorithm can stop and 
report failure if there are more than $c$ medium sized 
vectors in $\Lambda$.

\subsection{Unlabeled setting} In the unlabeled setting, the 
algorithm does not have access to $c$, and so there is a 
possibility that it reports an incorrect result.  Here we 
describe some, to our knowledge non-exhaustive, tests that 
can detect a failure to correctly compute a basis for 
$S^{\sigma_\p}$.

Seymour's algorithm may report that the matroid represented by 
the linear space $S$ obtained in the lattice reduction 
step is not graphic.  Since $S^{\sigma_\p}$ does represent 
a graphic matroid, whenever Seymour's algorithm fails, 
$S\neq S^{\sigma_\p}$.

Even if an incorrectly computed $S$ represents a graphic 
matroid, of a graph $H$ necessarily not isomorphic to $G$,
the orientation-finding step may fail.  If $H$ is not $2$-connected,
then the failure will occur when that step tries to find 
an ear-decomposition.  
If $H$ is $2$-connected, the algorithm 
may find an orientation 
$H^\tau$ of $H$ that is not  a DAG, and, 
hence, $\tau$ is not vertex consistent.

Finally, if the previous tests all pass, 
we can check that the output configuration 
$\q$ has length measurements consistent with $\l$.
If there is no noise, the check is exact: we can verify 
that each edge of the output graph $G$ has, in $\ell(\q)$, 
the same length as the corresponding entry of $\l$.  If not, 
the algorithm has failed.  When there is noise, the test is 
not exact, but a residual error in the measured lengths 
larger than that allowed by Theorem \ref{thm:main2b} indicates
the algorithm has failed.

While we expect these tests are useful in practice, 
the possibility that all of them pass for an 
incorrectly computed $S^{\sigma_\p}$ remains a (low probability)
possibility.

\section{Denser graphs}
The limiting factor of our algorithm is its need for highly accurate measurements. Here we describe an option 
for certain denser graphs, which can require many fewer bits of accuracy.
\begin{definition}
We say that $G$ has a  \defn{$k$-basis}, for some $k$, and orientation $\sigma$, the $\RR$-cycle space of 
$G^\sigma$ has a cycle induced basis $\{\w_{\rho,\sigma}\}$, where each of the  cycles $\rho$
has at most $k$ edges.
\end{definition}
If we assume that $G$ has a $k$-basis,  then, given a noisy measurement 
$\l = \ell(\p) + \eps$, with $\eps\in \{-1,0,1\}^m$ we can compute a 
basis for $\Lat(\l)$ by enumerating every vector 
\[
\hat{\w} = \begin{pmatrix}
    \w \\ e
\end{pmatrix}
\]
where $\w$ has at most $k$ non-zero entries, each of which is $\pm 1$ and $-k\le e\le k$
and checking whether it is in $\Lat(\l)$ using the lattice generating matrix $L(\l)$.  
Assuming that $G$ has a $k$-basis, a maximal linearly-independent subset of the vectors of this form 
will span a subspace of $\Lat(\l)$ that contains $\hat{S}^{\sigma_\p}$.
A similar analysis to that in Section \ref{sec:lll}, shows that the span of the vectors returned 
by the brute-force procedure is, with high probability over a uniformly selected $\p$, 
a basis for $\hat{S}^{\sigma_\p}$. Once we have recovered $\hat{S}^{\sigma_\p}$, we can proceed 
as if we had used LLL.

The complexity of the brute-force approach to find a $k$-basis for $S^\sigma$ is exponential 
in $k$, but polynomial in $m$.  If $k$ is fixed and independent of $m$, it is polynomial, 
albeit much slower than LLL for even modest $k$.  The advantage of the brute-force 
procedure is that, because it solves the shortest-vector problem for $\Lat(\l)$ exactly, 
we can take $b$ to be any quantity that is $\omega(m\log m)$, since we no longer 
have to account for the approximation factor in the LLL algorithm.

\section{Experiments}
\label{sec:exp}
We implemented the LLL step in our algorithm to explore how it performs on 
different graph families.  We looked at three family of graphs:
cycles; nearly $3$-regular graphs (with all vertices degree $3$ except 
for one which has degree $4$ or $5$); and complete graphs.  

\begin{figure}
    \centering
    \begin{center} \begin{tabular}{cccc}
  {  \includegraphics[scale=.4]{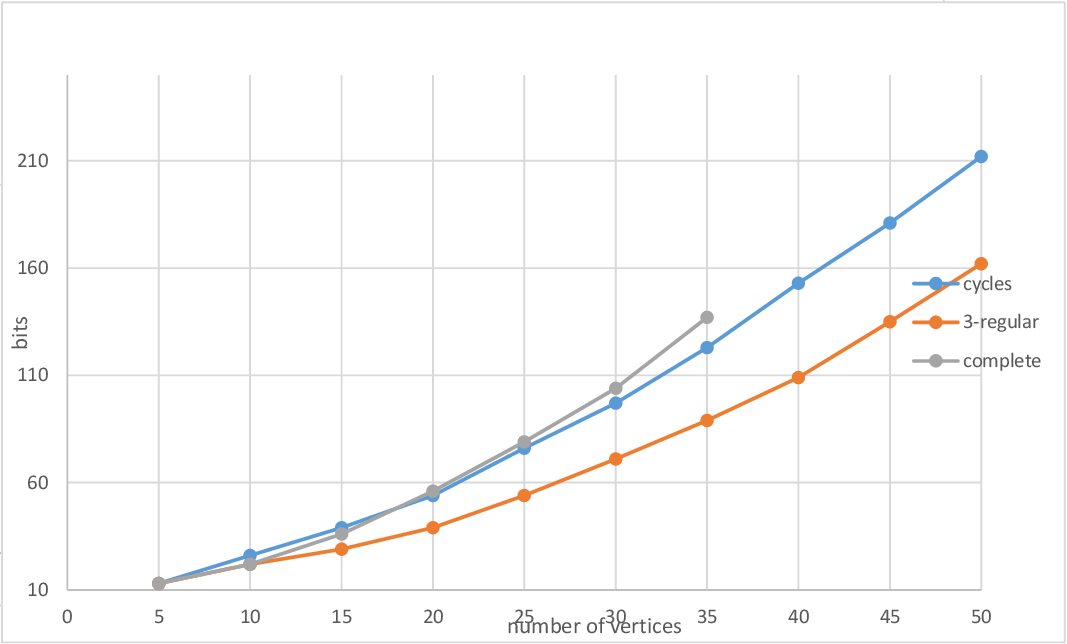}}&
  {  \includegraphics[scale=.4]{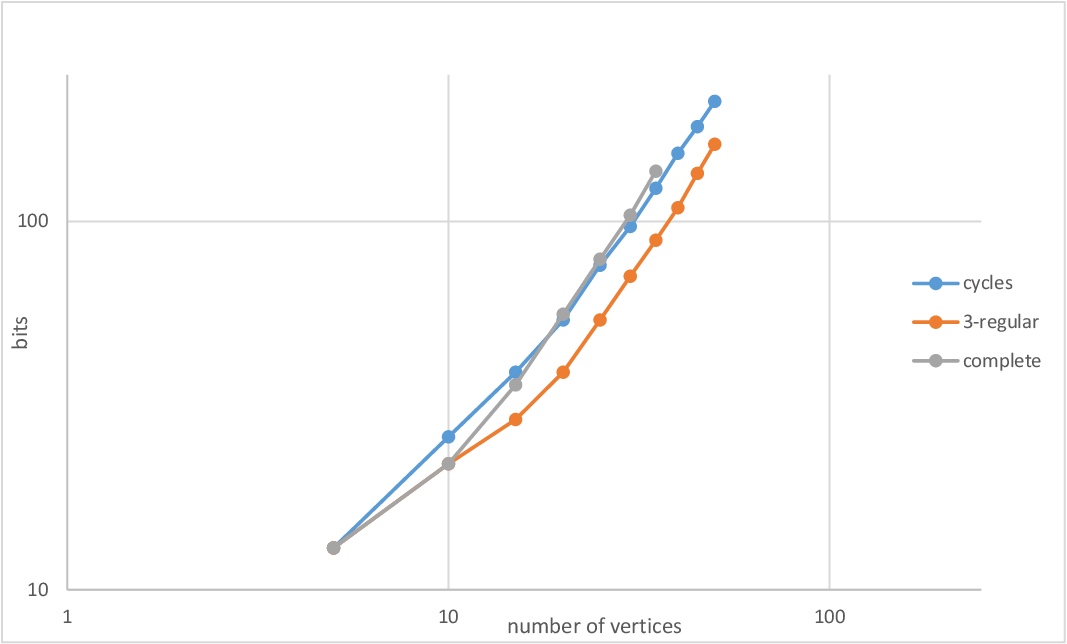}}
\end{tabular} \end{center}
    \caption{Number of bits required for the LLL step to 
    recover the cycle space with probability $0.9$ vs. number of 
    vertices. Log-Log on right.}
\label{fig:exp}
\end{figure}
For 
a fixed number of vertices (and hence edges), we used our implementation 
to search for the number of bits necessary for the probability (over random 
point sets) of 
the LLL step correctly identifying the cycle space of the graph to 
be at least $0.9$.  (Here the configurations were integer, and
the lengths were measured with $\{-1,0,1\}$ error.)
For the nearly $3$-regular graphs, we tried $5$ different 
random ensembles and used the maximum number of bits required for each 
size.
These experiments are ``optimistic'', since 
they use knowledge of the number of vertices to avoid using a
threshold value to identify small vectors in the reduced lattice 
basis. 

In these results, we see that 
for all $3$ families
of graphs, the 
bit requirement grows at about the rate $n^{1.5}$ (see the log-log plot of Figure~\ref{fig:exp}).
More bits are needed for the cycle graphs
than the nearly $3$-regular graphs. This is perhaps due to the 
fact that the nearly $3$-regular graphs have basis cycles, with
each cycle much smaller than $n$.
For the complete graphs, 
where $m=O(n^2)$, this gives us growth that
is somewhat sub-linear in $m$.
  All of this behaviour is significantly 
better than the pessimistic
$m^2$ bounds of our theory. On the other hand, even for moderate
sized graphs, the required bit accuracy quickly becomes higher than we would be able to get from a physical measurement system.

The data sets generated during and/or analysed during the current study are available from the corresponding author on reasonable request.

\end{document}